\documentclass{amsart}
\usepackage{amsmath,amsthm,amsfonts}
\DeclareMathOperator*{\bd}{bd}
\DeclareMathOperator{\grad}{grad}

\begin{document}
\title[Cooperative Systems with First Integral]{Cooperative
Irreducible Systems
\\
of Ordinary Differential Equations
\\
with First Integral}

\author{Janusz Mierczy\'nski}

\address{Institute of Mathematics, Wroc{\l}aw University of
Technology, Wybrze\.ze Wyspia\'nskiego 27, PL-50-370 Wroc{\l}aw,
Poland}

\dedicatory{Proceedings of the Second Marrakesh International
Conference on Differential Equations, 1995}

\begin{abstract}
This note considers cooperative irreducible systems of ordinary
differential equations admitting a $C^{1}$ first integral with
positive gradient.  We prove that all forward or backward
nonwandering points are equilibria.  We obtain also some results on
the global phase portrait of such systems.  The main tool of proof is
a canonically defined Finsler structure with respect to which the
derivative skew-product dynamical system is contractive.
\end{abstract}


\date{}

\maketitle

\newcommand{\Id}{\ensuremath{\mathrm{Id}}}
\newcommand{\reals}{\ensuremath{\mathbb{R}}}
\newcommand{\Rn}{\ensuremath{{\mathbb{R}}^{n}}}
\newcommand{\Rnmin}{\ensuremath{{\mathbb{R}}^{n-1}}}
\newcommand{\Rnplus}{\ensuremath{{\mathbb{R}}^{n}_{+}}}

\newtheorem{theorem}{Theorem}
\newtheorem{lemma}[theorem]{Lemma}
\newtheorem{remark}[theorem]{Remark}
\newtheorem{corollary}[theorem]{Corollary}
\renewcommand{\theequation}{\arabic{equation}}

\section{Introduction}
A system of ordinary differential equations (ODEs)
\begin{equation}
\label{eq}
\dot{x}^{i} = f^{i}(x), \qquad x = (x^{1}, \dots, x^{n}),
\end{equation}
where $f = (f^{1}, \dots, f^{n}) \colon X \to \Rn$ is a $C^{1}$
vector field on an open set $X \subset \Rn$, is called {\em
cooperative\/} if $({\partial}f^{i}/{\partial}x^{j})(x) \ge 0$  for
$i \ne j$ and all $x \in X$.

\smallskip
Let $\phi(t;x_{0})$ denote the nonextendible solution of system
(\ref{eq}) with the initial condition $\phi(0;x_{0}) = x_{0}$. We
write $\phi_{t}x_{0}$ instead of $\phi(t;x_{0})$.  For each $x \in X$
the mapping $t \mapsto \phi_{t}x$ (called the {\em trajectory\/} of
$x$) is defined on an open interval $(\sigma(x),\tau(x))$ containing
$0$. The restriction of the trajectory of $x$ to $(\sigma(x),0]$
[resp.~to $[0,\tau(x))$] is called the {\em backward\/} [resp.~{\em
forward\/}] {\em semitrajectory\/} of $x$.  The images of
(semi)trajectories are referred to as {\em
\textup{(}semi\textup{)}orbits\/}.  We say $x \in X$ is an {\em
equilibrium\/} if $\phi_{t}x = x$ for all $t \in \reals$, or,
equivalently, if $f(x) = 0$.  A point $y \in X$ is an {\em
$\omega$-limit point\/} of $x\in X$ if there is a sequence $t_{k} \to
\infty$ as $k \to \infty$, such that $\lim_{k\to\infty} \phi_{t_{k}}x
= y$. Notice that $\tau(x) = \infty$, while it is possible that
$\tau(y) < \infty$.  The definition of an $\alpha$-limit point is
analogous.  The set of $\omega$-limit [resp.~$\alpha$-limit] points
is called the {\em $\omega$-limit set \textup{[}$\alpha$-limit
set\textup{]}\/} of $x$, and denoted by $\omega(x)$ [$\alpha(x)$].

The symbol $\lVert \cdot \rVert$ stands for the Euclidean norm in
$\Rn$. We say $x \in X$ is {\em forward nonwandering\/} if for each
$\epsilon > 0$ and each $0 < t < \tau(x)$ there are $y \in X$ with
$\tau(y) > t$ and $t < \theta < \tau(y)$ such that $\lVert x-y \rVert
< \epsilon$ and $\lVert x - \phi_{\theta}y \rVert < \epsilon$.  A
point $x \in X$ is {\em backward nonwandering\/} if for each
$\epsilon > 0$ and each $\sigma(x) < t < 0$ there are $y \in X$ with
$\sigma(y) < t$, and $\sigma(y) < \theta < t$ such that $\lVert x - y
\rVert < \epsilon$ and $\lVert x - \phi_{\theta}y \rVert < \epsilon$.
A point that is forward nonwandering or backward nonwandering is
called {\em nonwandering\/}. It is straightforward that an
$\omega$-limit point is forward nonwandering.  Notice that in the
above definitions we do not assume the forward [backward]
semitrajectory of either $x$ or $y$ to be defined on the whole
half-line $[0,\infty)$ [$(-\infty,0]$].

For two points $x$, $y\in\Rn$ denote
\begin{eqnarray*}
x \le y & \quad \text{if } x^{i} \le y^{i} \text{ for each }i,
\\
x < y & \quad \text{if } x \le y \text{ and } x \ne y,
\\
x \ll y & \quad \text{if } x^{i} < y^{i} \text{ for each } i.
\end{eqnarray*}

For $x\le y$ we define a {\em closed order interval\/} as
\begin{equation*}
[x,y] := \{z \in \Rn: x \le z \le y\},
\end{equation*}
and for $x\ll y$ we define an {\em open order interval\/} as
\begin{equation*}
[[x,y]] := \{z \in \Rn: x \ll z \ll y\}.
\end{equation*}
A set $X\subset\Rn$ is said to be {\em p-convex\/} if the line
segment with endpoints $x$ and $y$ is contained in $X$ for each $x$,
$y \in X$, $x < y$, and {\em order-convex\/} if $[x,y] \subset X$ for
each $x$, $y \in X$, $x<y$.

The next result gives an important property of cooperative
systems of ODEs.
\begin{theorem}
\label{coop}
Assume \eqref{eq} is a cooperative system of ODEs on a p-convex open
set $X \subset \Rn$.  Let $x \le y$.  Then $\phi_{t}x \le \phi_{t}y$
for each $t \in [0, \min(\tau(x),\tau(y)))$.
\end{theorem}
The above theorem was proved in \cite{Mueller27} and \cite{Kamke32}.
Some gaps in the earlier proofs were filled in \cite{W50}. The
property is referred to as {\em monotonicity\/} of the local flow
generated by \eqref{eq}.

An important feature of cooperative systems of $n$ ordinary
differential equations is that the limiting behavior of a point whose
forward semiorbit has compact closure is at most so complicated as
that in a general system of $n-1$ ODEs.  More precisely, the
following result holds (see Theorem A in \cite{Moe82}):
\begin{theorem}
\label{coop-syst}
Let \eqref{eq} be a cooperative system of ODEs on a p-convex open set
$X\subset\Rn$.  Assume that the forward
\textup{[}resp.~backward\textup{]} semiorbit of $x \in X$ has compact
closure in $X$.  Put $L = \omega(x)$ \textup{[}resp.~$L =
\alpha(x)$\textup{]}. Then the restricted flow $\{\phi_{t}|L\}$ is
topologically equivalent to the flow of a Lipschitz system of ODEs on
$\Rnmin$.  Moreover, no two points in $L$ are related by $\ll$.
\end{theorem}

An important class of cooperative systems is formed by {\em
cooperative irreducible\/} systems of ODEs, that is, cooperative
systems such that for each $x\in X$ the matrix
$[({\partial}f^{i}/{\partial}x^{j})(x)]$ is irreducible. A system of
ODEs is called {\em strongly cooperative\/} if
$({\partial}f^{i}/{\partial}x^{j})(x) > 0$ for $i \ne j$.

For cooperative irreducible systems Theorem \ref{coop} can be
strengthened to:
\begin{theorem}
\label{coop_irr}
Assume \eqref{eq} is a cooperative irreducible system of ODEs on a
p-convex open set $X \subset \Rn$.  Let $x < y$. Then $\phi_{t}x \ll
\phi_{t}y$ for each $t \in (0,\min(\tau(x),\tau(y)))$.
\end{theorem}

The above property is called {\em strong monotonicity\/} of the local
flow generated by \eqref{eq}.

For cooperative irreducible systems one can say much more about their
behavior (see Theorem 2.4 in \cite{ST91}; for earlier results see
\cite{Moe88} and \cite{Pol89}):
\begin{theorem}
\label{generic}
Assume \eqref{eq} is a cooperative irreducible system of ODEs on a
p-convex open set $X \subset \Rn$ such that each forward semiorbit
has compact closure in $X$.  Then there exists an open dense set $Y
\subset X$ such that $\omega(x)$ is a singleton for each $x \in Y$.
\end{theorem}

Generally, in cooperative irreducible systems there are no
restrictions on $\alpha$-limit sets.  More precisely, according to
\cite{Sm76}, any dynamics on the standard $n$-dimensional simplex can
be embedded as a repeller in an $(n+1)$-dimensional strongly
cooperative system of ODEs.

For a recent monograph on cooperative systems of ODEs see
\cite{Hal95}.

\section{Cooperative Irreducible Systems with First Integral}

By a {\em first integral\/} for \eqref{eq} we mean a continuous
function $H \colon X \to \reals$ which is constant on orbits of
\eqref{eq}. A first integral is {\em nontrivial\/} if it is not
constant on any open set.

The existence of nontrivial first integrals puts severe restrictions
on cooperative irreducible systems.  Namely, it was proved in
\cite{Moe85} that if the set of equilibria is countable and all
points have forward semiorbit closure compact in $X$ then each first
integral is trivial.

In the three-dimensional case much more can be proved, even without
assuming the abundance of points with compact forward semiorbit
closure (see \break \cite{JM95}):
\begin{theorem}
\label{three-d}
Assume that \eqref{eq} is a cooperative irreducible system of ODEs on
$\Rn$ admitting a $C^{1}$ first integral with nonzero gradient.  Then
each limit set is either empty or a singleton.
\end{theorem}

\section{Cooperative Irreducible Systems with Monotone
First Integral}

\subsection{Limiting Behavior}

When one assumes that a first integral for a cooperative irreducible
system of ODEs is {\em strongly monotone\/}, that is, from $x<y$ it
follows $H(x) < H(y)$, then, under the assumption that all forward
semiorbits have compact closure in $X$, all $\omega$-limit sets are
singletons (compare \cite{JM87}).

This result carries over to the case of abstract strongly monotone
semidynamical systems with monotone first integral defined on
strongly ordered Banach spaces (\cite{A}).

Further, for some classes of cooperative periodic [resp.\ almost
periodic] (in time) systems of ODEs admitting a first integral with
appropriate monotonicity properties it was proved that each solution
with compact forward semiorbit closure converges to a periodic
[resp.\ almost periodic] solution (\cite{N79}, \cite{SN80},
\cite{Tang93}, \cite{Jifa95}).  Moreover, in the almost periodic case
in the corresponding nonlinear skew-product (local) flow the image of
that solution intersects each fiber at precisely one point (in the
case where the system is periodic with period $T$ this simply means
that the limiting solution has period $T$).

In many of the proofs of the results mentioned above the
idea was to use a kind of Lyapunov function.

\medskip
In the present note we do not assume any compactness of forward (or
backward) semiorbits.  On the other hand, we make extensive use of
the fact that a cooperative irreducible system of ODEs generates a
linear skew-product dynamical system on the tangent bundle of $X$
possessing some monotonicity properties.

\smallskip
Let us introduce some notation.  For $t \in (\sigma(x),\tau(x))$ the
derivative $D\phi_{t}(x)$ of $\phi_{t}(x)$ with respect to $x$ is a
linear isomorphism from the tangent space at $x$ into the tangent
space at $\phi_{t}x$, satisfying the nonautonomous linear matrix ODE
\begin{equation*}
M' = Df(\phi_{t}x)M
\end{equation*}
with initial condition $M(0) = \Id$, where $Df :=
[({\partial}f^{i}/{\partial}x^{j})]$.  The local linear skew-product
dynamical system
\begin{equation*}
(x,v) \mapsto (\phi_{t}x,D\phi_{t}(x)v), \quad x \in X, \
v \in \Rn, \ \sigma(x) < t < \tau(x),
\end{equation*}
will be referred to as the {\em derivative\/} local flow.

The order relations $\le$, $<$ and $\ll$ are defined in a natural way
on tangent vectors.  We will refer to vectors $v \ge 0$ as {\em
nonnegative\/}, and to vectors $v \gg 0$ as {\em positive\/}.  The
set of all nonnegative (free) vectors is called the {\em
\textup{(}nonnegative\textup{)} cone\/}.

\medskip
The derivative flow enjoys the following strong monotonicity property
(see \newline \cite{Moe85}):
\begin{theorem}
\label{skew-mon}
Assume that \eqref{eq} is a cooperative irreducible system of ODEs on
an open set $X \subset \Rn$.  Then for each $x \in X$, each $t \in
(0,\tau(x))$ and each $v > 0$ one has $D\phi_{t}(x)v \gg 0$.
\end{theorem}

Let $\mathcal{H}(x)$ stand for the level set of the first integral
$H$ passing through $x$.  If $H$ is of class $C^{1}$ and $\grad{H(x)}
\ne 0$ for all $x \in X$ then $\mathcal{H}(x)$ is a $C^{1}$
submanifold of codimension $1$.

The main result of the present note is

\begin{theorem}
\label{mt}
Let \eqref{eq} be a cooperative irreducible system of ODEs on an open
$X \subset \Rn$.  Assume that \eqref{eq} admits a first integral $H$
of class $C^{1}$ with positive gradient.  Then each nonwandering
point is an equilibrium.
\end{theorem}
\begin{proof}
In \cite{JM91} it was proved that for a cooperative irreducible
system of ODEs admitting a $C^{1}$ first integral with positive
gradient there is a canonical Finsler on the foliation of $X$ into
level sets of $H$ under which the derivative (local) flow is
contractive.  By a {\em Finsler\/} we understand a continuous mapping
$(x,v) \mapsto \lvert v \rvert_{x}$ such that for each $x$ the
assignment $v \mapsto \lvert v \rvert_{x}$ is a norm on the tangent
space at $x$ of $\mathcal{H}(x)$.

This canonical Finsler is constructed in the following way: For each
$x\in X$, we translate the tangent space of $\mathcal{H}(x)$ at $x$
by a positive vector $v_{x}$ such that $\langle \grad{H(x)}, v_{x}
\rangle = 1$.  The intersection of the resulting hyperplane with the
nonnegative cone is a compact convex set $A_{x}$ linearly isomorphic
to the standard $(n-1)$-dimensional simplex. Finally, we take the
compact convex balanced set $A_{x} - A_{x}$ to be the unit ball in
the tangent space of $\mathcal{H}(x)$ at $x$.

Proposition~2 in \cite{JM91} states that for each $x \in X$, each
nonzero vector $v$ tangent at $x$ to $\mathcal{H}(x)$ and each $t \in
(0,\tau(x))$ one has $\lvert D\phi_{t}(x)v \rvert_{\phi_{t}(x)} <
\lvert v \rvert_{x}$. (As a matter of fact, that proposition is
stated for strongly cooperative systems, but its proof carries over
{\em verbatim\/} to the case of cooperative irreducible systems.)

Suppose for contradiction that $x$ is a forward nonwandering point
not being an equilibrium.  Let $L \subset X$ be a $C^{1}$ embedded
$(n-1)$-dimensional disk transverse to $f(x)$ and having $x$ in its
relative interior.  Pick $s > 0$, $s < \min\{\tau(z): z \in L\}$.
Denote
\begin{equation*}
\lambda := \max\{\lvert D\phi_{s}(z)v \rvert_{\phi_{s}(z)}: z \in L,
\ v \in T_{z}\mathcal{H}, \ \lvert v \rvert_{z} = 1\},
\end{equation*}
where $T_{z}\mathcal{H}$ denotes the tangent space of
$\mathcal{H}(z)$ at $z$.  It is obvious that $0 < \lambda<1$.  Now,
take a $C^{1}$ embedded $(n-1)$-dimensional disk $M\subset L$
transverse to $f(x)$, having $x$ in its relative interior and such
that
\begin{equation*}
(1-\mu) \lvert f(x) \rvert_{x} < \lvert f(z) \rvert_{z} < (1+\mu)
\lvert f(x) \rvert_{x} \text{ for all } z \in M,
\end{equation*}
where $\mu := (1-\lambda)/(1+\lambda)$.

As $x$ is, by assumption, a forward nonwandering point, there exists
a point $z \in M$ and $t > s$ such that $\phi_{t}z \in M$.  But
\begin{equation*}
\lvert f(\phi_{t}z) \rvert_{\phi_{t}z} = \lvert
D\phi_{t-s}(\phi_{s}z)f(\phi_{s}z) \rvert_{\phi_{t}z} < \lvert
f(\phi_{s}z) \rvert_{\phi_{s}z} \le \lambda \lvert f(z) \rvert_{z} <
(1-\mu) \lvert f(x) \rvert_{x},
\end{equation*}
a contradiction.

For a backward nonwandering point $x$ not being an equilibrium, let
$L \subset X$ be a $C^{1}$ embedded $(n-1)$-dimensional disk
transverse to $f(x)$ and having $x$ in its relative interior.  Pick
$s < 0$, $s < \max\{\sigma(z): z \in L\}$.  Denote
\begin{equation*}
\lambda := \min\{\lvert D\phi_{s}(z)v \rvert_{\phi_{s}(z)}: z \in L,
v \in T_{z}\mathcal{H}, \lvert v \rvert_{z} = 1\}.
\end{equation*}
It is obvious that $\lambda > 1$.  Now, take a $C^{1}$ embedded
$(n-1)$-dimensional disk $M \subset L$ transverse to $f(x)$, having
$x$ in its relative interior and such that
\begin{equation*}
(1-\mu) \lvert f(x) \rvert_{x} < \lvert f(z) \rvert_{z} < (1+\mu)
\lvert f(x) \rvert_{x} \text{ for all } z \in M,
\end{equation*}
where $\mu := (\lambda-1)/(\lambda+1)$.

As $x$ is, by assumption, a backward nonwandering point,
there exists a point $z\in M$ and $t<s$ such that
$\phi_{t}z\in M$.  But
\begin{equation*}
|f(\phi_{t}z)|_{\phi_{t}z}=
|D\phi_{t-s}(\phi_{s}z)f(\phi_{s}z)|_{\phi_{t}z}
>|f(\phi_{s}z)|_{\phi_{s}z}\ge
\lambda|f(z)|_{z}>(1-\mu)|f(x)|_{x},
\end{equation*}
a contradiction.
\end{proof}
\medskip

As a corollary we obtain (see \cite{JM91})
\begin{theorem}
\label{t2}
Let \eqref{eq} be a cooperative irreducible system of ODEs on an open
$X \subset \Rn$ admitting a first integral $H$ of class $C^{1}$ with
positive gradient.  Then
\begin{enumerate}
\item[(i)]
Each $\omega$-limit set is either a singleton or empty.
\item[(ii)]
If $\alpha(x)$ is nonempty then $x$ is an equilibrium.
\end{enumerate}
\end{theorem}
\begin{proof}
Part (i) follows by the fact that each $\omega$-limit point is
forward nonwandering.  Also, each $\alpha$-limit point is backward
nonwandering.  Assume that for some $x$, $\alpha(x) = \{y\}$.  We
have $0 = \lvert f(y) \rvert_{y} = \lim_{t \to -\infty} \lvert
f(\phi_{t}x) \rvert_{\phi_{t}x} = \sup\{\lvert f(\phi_{t}x)
\rvert_{\phi_{t}x}: t \in (-\infty,\tau(x)) \}$ (see Proposition 2 in
\cite{JM91}), hence $\lvert f(x) \rvert_{x} = 0$, and $x$ is an
equilibrium, therefore $x = y$.
\end{proof}

\subsection{Global Picture}

In the present subsection we shall obtain some insight into the
global nature of the dynamical system restricted to a level set
$\mathcal{H}(x)$ of $H$.  The standing assumption will be:

{\em $X \subset \Rn$ is an open order-convex set such that for any
pair $x$, $y \in X$ their maximum $x \vee y$ and minimum $x \wedge y$
are in $X$.}

For instance, $X = \Rn$, or $X = [[a,b]]$ for some $a$, $b \in \Rn$,
$a \ll b$, or else $X$ is the {\em positive orthant\/} $\Rnplus :=
\{x \in \Rn: x^{i} > 0$ for all $i\}$.

We begin with an auxiliary
\begin{lemma}
\label{l1}
Let $H$ be a $C^{1}$ first integral with positive gradient for a
cooperative irreducible system \eqref{eq}.  Then for each $x \in X$,
$\mathcal{H}(x)$ is connected.
\end{lemma}
\begin{proof}
Fix $x \in X$, and write $\mathcal{H} := \mathcal{H}(x)$.  Take $y$,
$z \in \mathcal{H}$, $y\ne z$.  As $y+z = y \wedge z {} + {} y \vee
z$, it follows that the points $y$, $z$, $y\wedge z$ and $y\vee z$
belong to a two-dimensional affine subspace $V$.  We have $H(y \wedge
z) < H(x)$ and $H(y \vee z) > H(x)$.  For each $\lambda \in [0,1]$
denote by $A_{\lambda}$ the union of the line segment joining $y \vee
z$ with ${\lambda}y + (1-\lambda)z$ and the line segment joining $y
\wedge z$ with ${\lambda}y + (1-\lambda)z$.  The set
$A:=\bigcup_{\lambda\in[0,1]}A_{\lambda}$ is a two-dimensional
(analytic) submanifold-with-corners contained in $V$.  It is apparent
that the gradient of the restriction $H|A$ is everywhere nonzero. The
implicit function theorem yields that the set $A \cap \mathcal{H}$ is
$C^{1}$ diffeomorphic to the real interval $[0,1]$. This finishes the
proof.
\end{proof}
\medskip

As in the case of a Riemannian metric, for a $C^{1}$ curve $\gamma$
contained in $\mathcal{H}$ define the {\em length\/} of $\gamma$ as
\begin{equation*}
\ell(\gamma):=
\int\limits_{a}^{b}|\gamma'(s)|_{\gamma(s)}\,ds,
\end{equation*}
where $\gamma \colon [a,b] \to \mathcal{H}$ is a parametrization of
the curve. It is straightforward that the length does not depend on a
parametrization.  We define the {\em Finsler distance\/} $d(x,y)$
between two points $x$, $y \in \mathcal{H}$ as the infimum of the
lengths of all $C^{1}$ curves with endpoints $x$ and $y$.  The
$C^{1}$ manifold $\mathcal{H}$ together with the Finsler distance
$d(\cdot,\cdot)$ is a metric space.

\begin{theorem}
\label{t3}
Let a cooperative irreducible system \eqref{eq} admit a first
integral $H$ of class $C^{1}$ with positive gradient.  Assume that
$\mathcal{H}$ is a level set of $H$ such that $\tau(x) = \infty$ for
all $x \in \mathcal{H}$.  Then either
\begin{enumerate}
\item[(a)]
There is precisely one equilibrium $y$ in $\mathcal{H}$, and $y$
is a global attractor in $\mathcal{H}$;
\end{enumerate}
or
\begin{enumerate}
\item[(b)]
There is no equilibrium in $\mathcal{H}$, and for each $x \in
\mathcal{H}$ one has $\omega(x) = \emptyset$.
\end{enumerate}
\end{theorem}
\begin{proof}
As a consequence of Theorem~\ref{mt}, if there is no equilibrium in
$\mathcal{H}$ then $\omega(x) = \emptyset$ for all $x \in
\mathcal{H}$.  So, assume $y \in \mathcal{H}$ is an equilibrium.
First, we claim that $y$ is a unique equilibrium in $\mathcal{H}$.
Suppose {\em per contra\/} that there is another equilibrium $y_{1}$.
Put $z: = y \vee y_{1}$.  We have $z > y$, $z > y_{1}$.  By strong
monotonicity, for each $t$, $0 < t < \tau(z)$, one has $\phi_{t}z \gg
\phi_{t}y = y$ and $\phi_{t}z \gg \phi_{t}y_{1} = y_{1}$.  But this
implies that $\phi_{t}z \gg z$, that is, in the level set of $H$
passing through $z$ there are two points being in the $\ll$ relation,
which is impossible.

By Main Theorem in \cite{JM91}, $y$ is (locally) exponentially
asymptotically stable relative to $\mathcal{H}$. Consequently, the
set $A: = \{x \in \mathcal{H}: \omega(x) = \{y\}\}$ is relatively
open in $\mathcal{H}$.  Suppose by way of contradiction that $A \ne
\mathcal{H}$.  Since by Lemma~\ref{l1} $\mathcal{H}$ is connected,
the relative boundary $\bd_{\mathcal{H}}A$ is nonempty. Pick a point
$z \in \bd_{\mathcal{H}}A$.  As $z \notin A$, we have $\omega(z) =
\emptyset$ by Theorem~\ref{mt}.

Because $y$ is asymptotically stable in $\mathcal{H}$, by \cite{C}
there is a compact relative neighborhood $B$ of $y$ in $\mathcal{H}$
such that $B\subset A$ and $\phi_{t}B \subset B$ for all $t\ge0$.
Take $\epsilon>0$ so small that $\{x \in \mathcal{H}: d(x,y) \le 2
\epsilon\}$ is contained in the relative interior of $B$. Pick $x_{1}
\in A$ with $d(x_{1},z) < \epsilon$. Let $T > 0$ be such that
$d(\phi_{T}x_{1},y) < \epsilon$.  As $\phi_{T}z$ exists, we must have
$d(\phi_{T}x_{1},\phi_{T}z) < \epsilon$. Consequently,
$d(\phi_{T}z,y) < 2 \epsilon$.  But from this it follows that
$\phi_{T}z \in B \subset A$, hence $\omega(z) = \{y\}$. This
contradiction completes the proof.
\end{proof}

\medskip
If $X = \Rn$, a well-known condition guaranteeing $\tau(x) = \infty$
for each $x$ is the existence of positive constants $C_{1}$ and
$C_{2}$ such that $\lVert f(x) \rVert \le C_{1} \lVert x \rVert +
C_{2}$. For another condition see the following result.
\begin{theorem}
\label{t4}
Let a cooperative irreducible system \eqref{eq} defined on $\Rn$
admit a first integral $H$ of class $C^{1}$, such that all the
coordinates of $\grad{H(x)}$ are positive, bounded and bounded away
from zero, uniformly in $x \in X$. Then we have either
\begin{enumerate}
\item[(a)]
For each $x \in \Rn$, $\omega(x)$ is a singleton.  Moreover, the
set of equilibria is simply ordered by $\ll$;
\end{enumerate}
or
\begin{enumerate}
\item[(b)]
For each $x \in \Rn$, $\omega(x) = \emptyset$.
\end{enumerate}
\end{theorem}
\begin{proof}
We begin by showing that there is a constant $C > 0$ such that
$\lVert v \rVert \le C \lvert v \rvert_{x}$ for each $x \in \Rn$ and
each vector $v$ tangent at $x$ to $\mathcal{H}(x)$.  Fix $x \in \Rn$,
and put
\begin{equation*}
G(x) := \frac{\grad{H(x)}}{\lVert \grad{H(x)} \rVert^{2}}.
\end{equation*}
Recall that in the construction of the Finsler $\lvert \cdot \rvert$
the unit ball $B_{x}$ is defined as $A_{x} - A_{x}$, where $A_{x} :=
\{v \ge 0: \langle G(x), v \rangle = 1\}$.  In particular, $G(x) \in
A_{x}$. Fix $v \in B_{x}$, that is, $|v|_{x} \le 1$.  Write $v =
v_{1} - v_{2}$, where $v_{1}$, $v_{2} \in A_{x}$.  For $i = 1$, $2$,
put $w_{i} := v_{i} - G(x)$.  Of course, $v = w_{1} - w_{2}$.  As
$v_{i} \ge 0$, we have $-w_{i} = G(x) - v_{i} \le G(x)$.  On the
other hand, $\langle G(x), -w_{i} \rangle$ is easily seen to be zero.
Write $G(x) = (a_{1},\dots,a_{n})$, $-w_{1} = (b_{1},\dots,b_{n})$.
We have $a_{j} > 0$, $b_{j} \le a_{j}$ and $\sum_{j=1}^{n} a_{j}
b_{j} = 0$. Define
\begin{equation*}
c_{j} := \frac{1}{a_{j} \lVert \grad{H(x)} \rVert^{2}} - a_{j}.
\end{equation*}
It is straightforward that all $c_{j}$'s are positive and
bounded uniformly in $x\in\Rn$.  We claim that
\begin{equation*}
-c_{j} \le b_{j} \le c_{j}
\end{equation*}
for each $1\le j\le n$.  Indeed, suppose $b_{k} > c_{k}$ for some
$k$. We have then
\begin{equation*}
a_{k}b_{k} > \frac{1}{\lVert \grad{H(x)} \rVert^{2}} - a_{k}^{2} =
\lVert G(x) \rVert^{2} - a_{k}^{2} = \sum_{\substack{j=1\\ j \ne
k}}^{n} a_{j}^{2} \ge \sum_{\substack{j=1\\ j\ne k}}^{n} a_{j} b_{j},
\end{equation*}
hence $\sum_{j=1}^{n} a_{j} b_{j} > 0$, a contradiction.  The other
inequality is proved by a similar argument.  We have thus obtained
that $\lVert v \rVert \le \lVert w_{1} \rVert + \lVert w_{2} \rVert$
does not exceed a constant independent of $x \in \Rn$.

Now suppose by way of contradiction that $T := \tau(x) < \infty$ for
some $x \in \Rn$.  Then the improper integral $\int_{0}^{T} \lvert
f(\phi_{t}x) \rvert_{\phi_{t}x} \, dt$ is convergent, from which it
follows in a standard way that the finite limit $\lim_{t \to T}
\phi_{t}x$ exists, a contradiction.

Assume that there exists an equilibrium $y\in\Rn$.  We may assume $y
= 0$ and $H(0) = 0$.  Take a positive real number $r$.  Let $\epsilon
> 0$ be such that for each $x$ with $\lVert x \rVert \le \epsilon$ one has
$H(x) < r/2$. Since the coefficients of $\grad{H}$ are positive and
bounded away from zero, for each $x > 0$, $\lVert x \rVert =
\epsilon$, the half line $\{sx: s \ge 1 \}$ intersects the level set
$H^{-1}(r) := \{z \in \Rn: H(z) = r\}$ at precisely one point $M(x)$.
It is easy to see that the mapping $M$ is a homeomorphism of $\{x \in
\Rn: x \ge 0, \ \lVert x \rVert = \epsilon \}$ onto $\{z \in
H^{-1}(r): z \ge 0 \}$.  As the latter set is forward invariant, a
well-known application of Brouwer's fixed point theorem implies that
there exists an equilibrium in $H^{-1}(r)$.  An analogous argument
applies to the case $r<0$.  From Theorem~\ref{t3} we deduce that
$\omega(x)$ is a singleton for any $x \in \Rn$.  The fact that the
set of equilibria is simply ordered follows by Proposition 2.1 in
\cite{JM87}.
\end{proof}

\end{document}